\documentclass[11pt, reqno]{amsart}
\usepackage{amssymb, amstext, amscd, amsmath}
\usepackage[all,cmtip]{xy}
\xyoption{all}
\usepackage{color}
\usepackage{hyperref} 
\usepackage{xy}
\xyoption{all}

\theoremstyle{plain}
\newtheorem{theorem}{Theorem}[section]
\newtheorem{thm}[theorem]{Theorem}
\newtheorem{cor}[theorem]{Corollary}
\newtheorem{prop}[theorem]{Proposition}
\newtheorem{lem}[theorem]{Lemma}
\theoremstyle{definition}
\newtheorem{rem}[theorem]{Remark}

\newtheorem{defn}[theorem]{Definition}

\newtheorem{con}{Conclusion}
\newtheorem*{con*}{Conclusion} 

\newcommand{\bZ}{{\mathbb{Z}}}

\newcommand{\fh}{{\mathfrak{h}}}

\newcommand{\upchi}{{\raise.35ex\hbox{\ensuremath{\chi}}}}

\newcommand{\qforal}{\quad\text{for all}\quad}

\newcommand{\Ad}{\operatorname{Ad}}

\newcommand{\diag}{\operatorname{diag}}

\newcommand{\id}{{\operatorname{id}}}

\newcommand{\ca}{\mathrm{C}^*}

\newcommand{\ol}{\overline}

\newcommand{\aut}{{\rm Aut}}

\newcommand{\sym}{\operatorname{Sym}}
\newcommand{\Aff}{\operatorname{Aff}}

\begin{document}
\title[Affine Actions and the YBE]{Affine Actions and the Yang-Baxter Equation}
\author[D. Yang]
{Dilian Yang}
\address{Dilian Yang,
Department of Mathematics $\&$ Statistics, University of Windsor, Windsor, ON
N9B 3P4, CANADA} \email{dyang@uwindsor.ca}

\begin{abstract}
In this paper, the relations between the Yang-Baxter equation and affine actions are explored in detail.  
In particular, we classify solutions of the Yang-Baxter equations in two ways: (i) by their associated affine actions of
their structure groups on their derived structure groups, and (ii) by the C*-dynamical systems obtained from their associated 
affine actions. On the way to our main results, several 
other useful results are also obtained. 
\end{abstract}

\thanks{The author was partially supported by an NSERC Discovery grant.}

\date{}
\maketitle

The Yang-Baxter equation has been extensively studied in the literature since \cite{Yan67}.  It plays important roles not only in statistical mechanics,
but also in other areas, such as, quantum groups, link invariants, operator algebras, and the conformal field theory. In general, it is a rather challenging  
problem to find all solutions of the Yang-Baxter equation. Following a suggestion given in \cite{Dri92}, many researchers have done a lot of work on studying 
a special but important class of solutions, which are now known as set-theoretic solutions. 
See, for example, \cite{CJO14, CJdR10, ESS, GC, GM, LYZ00, Sol00, Yan16} to name just a few, and the references therein. 

The main aim of this paper is to explore the relations between the Yang-Baxter equation and affine actions on groups. The main ideas
behind here are motivated by \cite{ESS, LYZ00, Sol00}. The rest of this paper is organized as follows. 
In Section \ref{S:YBE}, we recall some necessary background on the Yang-Baxter equation which
will be needed later. In Section \ref{S:affine}, we first introduce affine actions and some related notions, then
associate to every solution of the Yang-Baxter equation a regular affine action of its structure group on its derived structure group   (Proposition \ref{P:YBEaff}), and finally describe two constructions of solutions to the Yang-Baxter equation via their associated affine actions.
Our main results of this paper are given in Section \ref{S:classification}. 
We classify injective solutions of the Yang-Baxter equation in terms of their associated affine actions (Theorem \ref{T:giso}). 
We further obtain a connection with C*-dynamical systems. It is shown that injective solutions can also be classified via their associated C*-dynamical systems (Theorem \ref{T:ds}). 
We end this paper with an appendix, which provides a commutation relation for semi-direct product of solutions to the Yang-Baxter equation 
determined by cycle sets, which might be useful in the future studies. 

 \section{The Yang-Baxter equation} 
 \label{S:YBE}

In this section, we provide some background on the Yang-Baxter equation which will be useful later. 

Let $X$ be a (non-empty) set, and $X^n:=\overbrace{X\times\cdots \times X}^n$ for $n\ge 2$. 

\begin{defn}
\label{D:YBE}
Let $R(x,y)=(\alpha_x(y), \beta_y(x))$ be a bijection on $X^2$. 
We call $R$ a \textit{set-theoretic solution of the Yang-Baxter equation} (abbreviated as \textit{YBE})
if
\begin{align}
\label{E:YBE}
R_{12}R_{23}R_{12}=R_{23}R_{12}R_{23}
\end{align}
on $X^3$, where $R_{12}=\id_X\times R$ and $R_{23}=R\times \id_X$. 
We often simply call $R$ a \textit{YBE solution on $X$}. Sometimes, we write it as 
$R_X$ or a pair $(R,X)$. A YBE solution $R$ on $X$ is said to be  
\begin{itemize}
\item \textit{involutive} if $R^2=\id_{X^2}$;

\item  \textit{non-degenerate} if, for all $x\in X$, $\alpha_x$ and $\beta_x$ are bijections on $X$;

\item \textit{symmetric} if $R$ is involutive and non-degenerate. 
\end{itemize}
\end{defn}

\subsection*{Standing assumptions:} \textsf{All YBE solutions in the rest of this paper are always assumed to be set-theoretic and non-degenerate.}

\subsection{Two characterizations of YBE solutions} 
The following lemma is well-known in the literature and also easy to prove. 

 \begin{lem}
 \label{L:basic}
 Let $R(x,y)=(\alpha_x(y),\beta_y(x))$. Then $R$ is a YBE solution on $X$, if and only if the following properties hold true:
 for all $x,y,z\in X$, 
\begin{itemize}
\item[(i)]
$\alpha_x \alpha_y=\alpha_{\alpha_x(y)}\alpha_{\beta_y(x)},$
\item[(ii)]
$\beta_y\beta_x=\beta_{\beta_y(x)}\beta_{\alpha_x(y)},$ and 
\item[(iii)] 
$\beta_{\alpha_{\beta_y(x)}(z)} (\alpha_x(y))=\alpha_{\beta_{\alpha_y(z)}(x)} (\beta_z(y))$ {\rm (Compatibility Condition)}.
\end{itemize}
Furthermore, $R$ is involutive if and only if 
\[
\alpha_{\alpha_x(y)}(\beta_y(x))=x \quad\text{and}\quad \beta_{\beta_y(x)}(\alpha_x(y))=y\qforal x,y\in X.
\]
 \end{lem}

Let us associate to a given YBE solution an important object -- its structure group. 
\begin{defn}
Let $R(x,y)=(\alpha_x(y),\beta_y(x))$ be a YBE solution on $X$. 
The structure group of $R$, denoted as $G_{R_X}$, is the group generated by $X$ with commutation relations determined by $R$:
\[
G_{R_X}={}_{\text{gp}}\big\langle X; xy=\alpha_x(y)\beta_y(x)\text{ for all } x,y\in X\big\rangle.
\]
Sometimes we also write $G_{R_X}$ as $G_{R,X}$ or $G_{X}$. 
\end{defn}


One can easily rephrase the characterization given in Lemma \ref{L:basic} in terms of actions of structure groups
(cf., e.g., \cite{ESS, GM}). 

\begin{cor}
\label{C:char}
A map $R(x,y)=(\alpha_x(y), \beta_y(x))$ is a YBE solution on $X$, if and only if
\begin{itemize}
\item[(i)] $\alpha$ can be extended to a left action of $G_{R_X}$ on $X$, 
\item[(ii)] $\beta$ can be extended to a right action of $G_{R_X}$ on $X$, and 
\item[(iii)] the compatibility condition in Lemma \ref{L:basic} (iii) holds. 
\end{itemize}
 \end{cor}


\subsection{Constructing YBE solutions from old to new}
There are several known constructions of YBE solutions from old to new. For our purpose, we only introduce two below. 
The first one seems to be overlooked in the literature.
 
 \subsection*{$\blacktriangleright$ Dual of $R$} Let $R(x,y)=(\alpha_x(y),\beta_y(x))$ be a YBE solution on $X$. Define $R^\circ$ on $X^2$ by 
 \[
 R^\circ(x,y)=(\beta_x(y),\alpha_y(x))\qforal x,y\in X.
 \]
We call $R^\circ$ the \textit{dual of $R$}. It is also a YBE solution on $X$. Indeed,
 this can be seen by switching $x$ and $y$ in the first two identities, and $x$ and $z$ in the third one in Lemma \ref{L:basic}.
 We give it such a name because we `dualize' the process $xy=\alpha_x(y)\beta_y(x)$ in $G_{R_X}$ via
 $y\circ x=\beta_y(x)\circ \alpha_x(y)$ (by switching the factors on both sides).
 
 Clearly, $R^{\circ\circ}=R$.
 
 Let $\Phi: G_{R_X}\to G_{R^\circ_X}$ be defined via
 $\Phi(x):=x$ for $x\in X$ and $\Phi(xy):=y\circ x$ for all $x,y \in X$. Since $\Phi(xy)=\Phi(\alpha_x(y)\beta_y(x))$ for all $x,y\in X$, 
 $\Phi$ can be extended to an anti-isomorphism from $G_{R_X}$ to $G_{R^\circ_X}$.

 \subsection*{$\blacktriangleright$ Derived solution of $R$ \cite{ESS,Sol00}}
 Let $R(x,y)=(\alpha_x(y),\beta_y(x))$ be a YBE solution on $X$. 
Then 
\[
(x,y)\stackrel{R}\mapsto (\alpha_x(y),\beta_y(x))\stackrel{R}\mapsto \big(\alpha_{\alpha_x(y)}(\beta_y(x)), \beta_{\beta_y(x)}(\alpha_x(y))\big)
\] 
determines a  YBE solution
 \[
 (x,\alpha_x(y))\mapsto \big(\alpha_x(y), \alpha_{\alpha_x(y)}(\beta_y(x))\big),
 \] 
 namely,
 \[
 R': (x,y)\mapsto \big(y, \alpha_y(\beta_{\alpha_x^{-1}(y)}(x))\big).
 \] 
This solution $R'$ is called the \textit{derived solution of $R$}.

The \textit{derived structure group $A_{R_X}$ of $R$} is defined as
\[
A_{R_X}=\big\langle X: x\bullet y=y\bullet\alpha_y(\beta_{\alpha_x^{-1}(y)}(x))\text{ for all } x,y\in X
\big\rangle. 
\]
As $G_{R_X}$, $A_{R_X}$ is sometimes also written as $A_{R,X}$ or $A_X$. 

\begin{rem}
It is often useful to think that $A_{R_X}$ and $G_{R_X}$ have the same generator set $X$ with the relations
\[
x\bullet y=x\alpha_{x^{-1}}(y)\qforal x,y\in X,
\]
equivalently,
$
xy=x\bullet \alpha_x(y)
$
for all $x,y\in X$. 
\end{rem}

\begin{rem} 
\label{R:psi}

(i) If $R(x,y)=(\alpha_x(y), x)$, then $R'(x,y)=(y, \alpha_y(x))$. Namely, $R'=R^\circ$. 

(ii) As in \cite{Sol00}, one can also define another derived YBE solution 
\[
{}'\!R(x,y) =\big (\beta_x(\alpha_{\beta_y^{-1}(x)}(y)),x\big)\qforal x,y\in X.
\]

(iii) One can easily check the following:
 $R$ is symmetric $\Leftrightarrow$ ${}'\!R=R'$ $\Leftrightarrow$ $R^\circ$ is symmetric $\Leftrightarrow$ $A_{R_X}$ and $A_{R'_X}$ are abelian $\Leftrightarrow$ 
 $A_{R_X}$ and $A_{R^\circ_X}$ are abelian. 

\end{rem}

\subsection{A distinguished action of $G_{R_X}$ on $A_{R_X}$} 
Let $R(x,y)=(\alpha_x(y),\beta_y(x))$ be a YBE solution on $X$. 
By Corollary \ref{C:char}, both $\alpha$ and $\beta^{-1}$ can be extended to actions of $G_X$ on $X$. 
For our convenience, let 
\begin{align*}
\label{E:phi}
\phi_R(x,y)&:=\alpha_y(\beta_{\alpha_x^{-1}(y)}(x)),\\
\psi_R(x,y)&:= \beta_x(\alpha_{\beta_y^{-1}(x)}(y))  
\end{align*}
for all $x,y\in X$. 
Similar to \cite[Theorem 2.3]{Sol00}, one has the following. 

\begin{lem}
\label{L:inv}
$\phi$ is $G_{R_X}$-equivariant with respect to the action $\alpha$:
\[
\phi_R(\alpha_g(x),\alpha_g(y))=\alpha_g(\phi_R(x,y)) \qforal x,y\in X \text {and } g\in G_{R_X}.
\]
\end{lem}

\begin{proof}
Notice that $\psi_{R^\circ}(x,y)=\phi_R(y,x)$ for all $x,y\in X$. Now first apply \cite[Theorem 2.3]{Sol00} to $R^\circ$ and then use 
the relation between $R$ and $R^\circ$ to obtain the following:
\begin{align*}
&\ \alpha_{g^{-1}}\psi_{R^\circ}(x,y)
=\psi_{R^\circ}(\alpha_{g^{-1}}(x),\alpha_{g^{-1}}(y))\qforal x,y\in X, g\in G_{R^\circ_X}\\
\Rightarrow  &\  \alpha_{g^{-1}}\phi_{R}(y,x)
=\phi_{R}(\alpha_{g^{-1}}(y),\alpha_{g^{-1}}(x))\qforal x,y\in X, g\in G_{R^\circ_X}\\
\Rightarrow &\ \alpha_{g}\phi_{R}(y,x)
=\phi_{R}(\alpha_{g}(y),\alpha_{g}(x))\qforal x,y\in X, g\in G_{R_X}.
\end{align*}
We are done. 
\end{proof}

Let $\aut_X(A_{R_X})$ be the group of all automorphisms of $A_{R_X}$ preserving $X$. 

\begin{prop}[\cite{Sol00}]
\label{P:inv}
Keep the above notation. 
The action $\alpha$ of $G_{R_X}$ on $X$ induces an action of $G_{R_X}$ on $A_{R_X}$ preserving $X$. That is, there is a group homomorphism from 
$G_{R_X}$ to $\aut_X(A_{R_X)}$.
\end{prop}

\begin{proof}
Notice that for all $g\in G_{R_X}$
\begin{align*}
&\ x\bullet y=y\bullet \phi_R(x,y)\\
\Rightarrow &\ \alpha_g(x)\bullet \alpha_g(y)=\alpha_g(y)\bullet \phi_R(\alpha_g(x),\alpha_g(y))\ (\text{replacing } x,y\text{ by }\alpha_g(x), \alpha_g(y))\\
\Rightarrow &\ \alpha_g(x)\bullet \alpha_g(y)=\alpha_g(y)\bullet \alpha_g(\phi_R(x,y))\ (\text{by Lemma }\ref{L:inv}).
\end{align*}
This implies that $\alpha_g$ can be extended to an element in $\aut_X(A_{R_X})$, as desired. 
\end{proof}

By Proposition \ref{P:inv}, one has an action $\alpha: G_{R_X}\curvearrowright A_{R_X}$. 


\section{Affine actions on groups}

\label{S:affine}

For a given YBE solution, we associate to it a regular affine action (Proposition \ref{P:YBEaff}). This plays a 
vital role in Section \ref{S:classification}. Conversely, in Subsection \ref{SS:constructing}, we use the two constructions of affine actions described 
in Subsection \ref{SS:aff}  to construct new YBE solutions.

Let $A$ be a group. Denote by $\Aff(A)$ the semi-direct product 
\[
\Aff(A)= \aut(A)\ltimes A,
\]
where $(S,a)(T,b)=(ST,aS(b))$ for all $S,T\in \aut(A)$ and $a,b\in A$. $\Aff(A)$ acts on $A$ via $(S,a)b=aS(b)$.

\begin{defn}
Let $G$ and $A$ be groups. 
An affine action of $G$ on $A$ is a group homomorphism $\rho: G\to \Aff(A)$.
\end{defn}

By definition, any affine action $\rho: G\to \Aff(A)$ has the following form: 
\[
\rho_g(a)=b(g)\pi_g(a)\qforal g\in G\text{ and } a\in A,
\] 
where $\pi:G\to \aut(A)$ is a group homomorphism,  called the \textit{linear part} of $\rho$, and 
$b:G\to A$, called the \textit{translational part} of $\rho$, is a 1-cocycle with respect to $\pi$ in coefficient $A$:
\[
b(g_1g_2)=b(g_1)\pi_{g_1}(b(g_2))\qforal g_1,g_2\in G.
\]
We sometimes simply write $\rho=(\pi,b)$, and also write $b(g)$ as $b_g$ for convenience. 

Recall that a group action is called \textit{regular} if it is transitive and free.  

The following lemma should be known. But we include a proof below for completeness. 

\begin{lem}
An affine action $\rho=(\pi,b)$ of a group $G$ on a group $A$ is regular, if and only if $b$ is bijective. 
\end{lem}

\begin{proof} 
$(\Rightarrow)$: Since $\rho$ is regular, for arbitrary $x,y$ in $A$ there is a unique $g\in G$ such that $\rho_g(x)=y$. 
Letting $x=e$ and $y\in A$ arbitrary shows that $b$ is surjective. 

Now suppose that $b(g_1)=b(g_2)$ for some $g_1,g_2\in G$. Then $\rho_{g_1}(e)=b(g_1)\pi_{g_1}(e)=b(g_2)\pi_{g_2}(e)=\rho_{g_2}(e)$. So $g_1=g_2$ as $\rho$ is free. 
Thus $b$ is injective.  

$(\Leftarrow)$: Let $x,y\in A$. Since $b$ is bijective, there is a unique $h_0\in G$ such that $b(h_0)=x$, and further a unique $g\in G$ such that $b(gh_0)=y$.
Then $\rho_g(x)=b(g)\pi_g(x)=b(g)\pi_g(b(h_0))=b(gh_0)=y$. Thus $\rho$ is transitive.

To show that $\rho$ is free, suppose that there are $g_1,g_2\in G$ such that  $\rho_{g_1}(x)=\rho_{g_2}(x)$ for some $x\in A$.  Then 
$b(g_1)\pi_{g_1}(x)=b(g_2)\pi_{g_2}(x)$. Since $b$ is surjective, there is $g\in G$ such that $b(g)=x$. Hence 
$b(g_1)\pi_{g_1}(b(g))=b(g_2)\pi_{g_2}(b(g))$, i.e., $b(g_1g)=b(g_2g)$. But $b$ is injective, $g_1g=g_2g$ and so $g_1=g_2$. 
Therefore, $\rho$ is free. 
\end{proof}


\begin{defn}
\label{D:conj}
Let $\rho^i$ be an affine action of a group $G$ on a group $A_i$ ($i=1,2$).  A group homomorphism $\varphi:A_1\to A_2$ is said to be $G$-equivariant
relative to $(\rho^1,\rho^2)$ if 
\begin{align}
\label{E:rho}
\varphi\circ\rho^1_g=\rho^2_g \circ \varphi\qforal g\in G.
\end{align}
That is, for every $g\in G$, the following diagram commutes:
\begin{align*}
		\xymatrix{
		A_1 \ar[r]^{\rho^1_g} \ar[d]_{\varphi} & A_1\ar[d]^{\varphi}&   \\
		A_2 \ar[r]_{\rho^2_g} &A_2.  &
		}
\end{align*}
If, furthermore, the above $f$ is bijective, then $\rho^1$ and $\rho^2$ are said to be conjugate. 
\end{defn}


\begin{rem}
\label{R:equ}
(i)
Let $\rho^i=(\pi^i,b^i)$ ($i=1,2$). 
It is easy to see that \eqref{E:rho} is equivalent to 
\begin{align*}
\varphi \circ\pi^1_g&=\pi^2_g\circ \varphi,\\
b^2_g&=\varphi\circ b^1_g
\end{align*}
for all $g\in G$. So, in particular, $\varphi$ is also $G$-equivariant relative to $(\pi^1,\pi^2)$. 

(ii) If $b^1$ is surjective, then using the definition of 1-cocycles, it is easy to see that the second identity in (i) above determines the first one. In fact, from the second one has for all $g,h\in G$
\begin{align*}
b^2_{gh}=\varphi(b^1_{gh})
\Rightarrow &\ b^2_g\pi_{g}^2(b^2_h)=\varphi(b^1_g)\varphi(\pi_g^1(b^1_h))\\
\Rightarrow &\ \pi_{g}^2(b^2_h)=\varphi(\pi_g^1(b^1_h))\ (\text{as }b_g^2=\varphi(b_g^1))\\
\Rightarrow &\ \pi_{g}^2(\varphi(b^1_h))=\varphi(\pi_g^1(b^1_h))\ (\text{as }b_h^2=\varphi(b_h^1))\\
\Rightarrow &\ \pi_{g}^2\circ \varphi=\varphi\circ \pi_g^1\ (\text{as } b^1(G)=A_1).
\end{align*}
\end{rem}

\subsection{Affine actions associated to YBE solutions}
This subsection shows why we are interested in affine actions. 

\begin{prop}[and \textbf{Definition}]
\label{P:YBEaff}
Any YBE solution $R$ on $X$ induces a regular affine action $\rho^X$ of $G_X$ on $A_X$. 

The action $\rho^X$ is called the {\rm affine action associated to $R_X$}, also denoted as $\rho_{R_X}$ or 
even just $\rho$ if the context is clear. 
\end{prop}

\begin{proof}
The proof is completely similar to \cite[Theorem 2.5]{Sol00}. We only sketch it here. By Proposition \ref{P:inv}, there is an action $\alpha: G_X\curvearrowright A_X$. 

\underline{Step 1}: Extend the mapping 
\[
\rho: X\to G_X\ltimes_\alpha A_X, \ x\mapsto (x,x)
\]
to a group homomorphism 
\[
\rho_G: G_X\to G_X\ltimes_\alpha A_X.
\]

To do so, one needs to check that 
\[
\rho(x)\rho(y)=\rho(\alpha_x(y))\rho(\beta_y(x))\qforal x,y\in X.
\]
But 
\[
\rho(x)\rho(y)=(x,x)(y,y)=(xy,x\bullet \alpha_x(y)),
\]
and similarly  
\[
\rho(\alpha_x(y))\rho(\beta_y(x))=\big(\alpha_x(y)\beta_y(x), \alpha_x(y)\bullet\alpha_{\alpha_x(y)}(\beta_y(x))\big).
\]
They are obviously equal. 

\underline{Step 2}: 
Let $p: G_X\ltimes_\alpha A_X\to A_X$ be the second projection to $A_X$, and let $b:=p\circ \rho_G$.
Then $b$ is a 1-cocycle with respect to the action $G_X\stackrel{\alpha}\curvearrowright A_X$:
\[
b(gh)=b(g)\bullet \alpha_g(b(h)) \qforal g,h\in G_X.
\]
In fact, for all $g,h\in X$,
\[
b(gh)=p\rho_G(gh)=p\rho((g,g)(h,h))=p(gh,g\bullet \alpha_g(h))=g\bullet\alpha_g(h),
\]
and 
\[
b(g)\bullet \alpha_g(b(h))=p\rho_G(g)\bullet \alpha_g(p\rho_G(h))=g\bullet \alpha_g(h).
\]

\underline{Step 3}: Check that $b$ is bijective (cf. \cite[Theorem 2.5]{Sol00}).
\end{proof}

\begin{rem}
In the sequel, we will frequently use that the simple fact that $b(x)=x$ for all $x\in X$ in the associated affine action $\rho^X=(\alpha,b)$ obtained from 
Proposition \ref{P:YBEaff}. 
\end{rem}

\subsection{Two constructions of affine actions}
\label{SS:aff}

In this subsection, we construct two new affine actions from given ones.

\subsection*{$1^\circ$ Lifting}
This generalizes a construction given in \cite{Bac14, BCJ15}, which plays key roles there.

Let $A$ and $H$ be two groups, and $\theta:H\to A$ a homomorphism.   
Suppose that $\rho=(\pi,b)$ is a regular affine action of $G$ on $A$, and that $\sigma$ is an action of $G$ on $H$,
such that $\theta$ is $G$-equivariant relative to $(\sigma, \pi)$:
\[
\theta\circ\sigma_g=\pi_g \circ\theta\qforal g\in G. 
\]
Introduce a new multiplication $\cdot$ on $H$ via 
\begin{align}
\label{E:mul}
x\cdot y:=x\sigma_{b^{-1}\circ\theta(x)}(y) \qforal x,y\in H. 
\end{align}
Then the \textit{lifting of $\rho$ from $A$ to $H$} is defined as 
\[ 
\tilde\rho_x(y)=x\cdot y \qforal x,y\in H. 
\]

\begin{con}
The lifting $\tilde \rho$ is an affine action of $(H,\cdot)$ on $H$. Furthermore, $\theta$ is $(H,\cdot)$-equivariant relative to $(\tilde \rho,  \rho\circ b^{-1}\circ\theta)$. 
\end{con}

Pictorially, one can summarize the above as follows: for all $g\in G$ and $h\in (H,\cdot)$
\begin{align*}
		\xymatrix{
		H \ar[r]^{\sigma_g} \ar[d]_{\theta} & H\ar[d]^{\theta}&   \\
		 A\ar[r]_{\pi_g} &A &
		}
\quad
\xymatrix{
		&\rightsquigarrow&\\
		&\rightsquigarrow&\\
		&&
		}
\quad
\xymatrix{
		H \ar@{-->}[rr]^{\tilde\rho_h} \ar[d]_{\theta} && H\ar[d]^\theta&   \\
		 A\ar[rr]_{\rho_{b^{-1}\circ\theta(h)}} &&A&\\  
		}
\end{align*}

\begin{proof} 
One can show that $(H, \cdot)$ is indeed a group: $\cdot$ is closed and associative, the identity is (still) $e$, and the inverse of $x$ in $(H,\cdot)$ is $\sigma_{b^{-1}\circ\theta(x)}(x)$. 
The verification is tedious and left to the reader. 

Also, $\tilde \rho$ is an affine action of $(H,\cdot)$ on $H$. In fact, 
\begin{align*}
\tilde \rho_{x_1\cdot x_2}(y)
&=(x_1\cdot x_2)\sigma_{b^{-1}\circ\theta(x_1\cdot x_2)}(y)\\
&=x_1\sigma_{b^{-1}\circ\theta(x_1)}(x_2)\sigma_{b^{-1}\circ\theta(x_1\sigma_{b^{-1}\circ\theta(x_1)}(x_2))}(y)\ (\text{by }\eqref{E:mul})\\
&=x_1\sigma_{b^{-1}\circ\theta(x_1)}(x_2)\sigma_{b^{-1}(\theta(x_1) \theta(\sigma_{b^{-1}\circ\theta(x_1)}(x_2)))}(y)\ (\text{as }{\theta} \text{ is a homomorphism})\\
&=x_1\sigma_{b^{-1}\circ\theta(x_1)}(x_2)\sigma_{b^{-1}(\theta(x_1) \pi_{b^{-1}\circ\theta(x_1)}\theta(x_2)))}(y)\ (\text{as }{\theta} \text{ is }G\text{-equivariant})\\
&=x_1\sigma_{b^{-1}\circ\theta(x_1)}(x_2)\sigma_{b^{-1}(\theta(x_1))b^{-1}\circ\theta(x_2)}(y)\ (\text{as } b \text{ is a 1-cocycle w.r.t }\pi)\\ 
&=x_1\sigma_{b^{-1}\circ\theta(x_1)}(x_2\sigma_{b^{-1}\circ\theta(x_2)}(y))\ (\text{as } \theta \text{ is an action})\\
&=x_1\sigma_{b^{-1}\circ\theta(x_1)}(\rho_{x_2}(y))\ (\text{by }\eqref{E:mul})\\
&=\tilde\rho_{x_1}(\tilde\rho_{x_2}(y))\ (\text{by }\eqref{E:mul}).
\end{align*}

Furthermore, $\rho\circ b^{-1}\circ \theta$ is an affine action of $(H,\cdot)$ on $A$. For this, since 
$\theta$ is $G$-equivariant for $(\sigma,\pi)$ and $b$ is 1-cocycle with respective to $\pi$, one has 
\begin{align*}
b^{-1}\circ\theta(h_1\cdot h_2)
&=b^{-1}(\theta(h_1)\theta\circ\sigma_{b^{-1}(\theta(h_1))}(h_2))\\
&=b^{-1}(\theta(h_1)\pi_{b^{-1}(\theta(h_1)}(\theta(h_2)))\\
&=b^{-1}\circ \theta(h_1)b^{-1}\circ \theta(h_2).
\end{align*}
Hence, for all $h_1,h_2\in H$ and $a\in A$, we get 
\begin{align*}
\rho_{b^{-1}\circ\theta(h_1\cdot h_2)}(a)
=b(b^{-1}\circ \theta(h_1)b^{-1}\circ \theta(h_2))\pi_{b^{-1}\circ \theta(h_1)b^{-1}\circ \theta(h_2)}(a)
\end{align*}
and 
\begin{align*}
&\ \rho_{b^{-1}\circ \theta(h_1)}\rho_{b^{-1}\circ\theta(h_2)}(a)\\
=&\ \rho_{b^{-1}\circ \theta(h_1)}(b(b^{-1}\circ\theta(h_2)\pi_{b^{-1}\circ\theta(h_2)}(a)\\
=&\ b(b^{-1}\circ \theta(h_1))\pi_{b^{-1}\circ\theta(h_1)}(b(b^{-1}\circ\theta(h_2))\pi_{b^{-1}\circ\theta(h_2)}(a))\\
=&\ b(b^{-1}\circ \theta(h_1))\pi_{b^{-1}\circ\theta(h_1)}(b(b^{-1}\circ\theta(h_2)))\pi_{b^{-1}\circ\theta(h_1)b^{-1}\circ\theta(h_2)}(a).
\end{align*}
This implies 
\[
\rho_{b^{-1}\circ\theta(h_1\cdot h_2)}=\rho_{b^{-1}\circ \theta(h_1)}\rho_{b^{-1}\circ\theta(h_2)}
\]
as $b$ is a 1-cocycle with respect to $\pi$. 

Using the property that $\theta$ is $G$-equivariant relative to $(\sigma,\pi)$ again, we have for all $x,z\in H$
\begin{align*}
\theta (\tilde \rho_z(x))
&=\theta(z\sigma_{b^{-1}\circ\theta(z))}(x)=\theta(z)\theta(\sigma_{b^{-1}\circ\theta(z)}(x))\\
&=b b^{-1}(\theta(z))\pi_{b^{-1}\circ\theta(z)}(\theta(x))\\
&=\rho_{b^{-1}\circ\theta(z)}(\theta(x)).
\end{align*}
Thus  
$
\theta \circ \tilde \rho_z=\rho_{b^{-1}\circ\theta(z)}\circ \theta,
$
as desired. 
\end{proof}


\subsection*{$2^\circ$ Semi-direct product} 

Let $\rho$  be an affine action of $G$ on $A$, and $\tilde\rho=(\tilde\pi, \tilde b)$ be a regular affine action of $\tilde G$ on $\tilde A$. Suppose $\theta:  G\curvearrowright \tilde G$ 
is an action of $G$ on $\tilde G$ such that 
\begin{align}
\label{E:asemi}
\theta_g(\tilde b^{-1}\tilde\pi_{h}\tilde b)=(\tilde b^{-1}\tilde \pi_{\theta_g(h)}\tilde b) \theta_g \qforal g\in G, h\in \tilde G. 
\end{align}

Then the \textit{semi-direct product of $\rho$ and $\tilde \rho$ via $\theta$} is defined as 
\begin{align*}
\rho\ltimes_\theta\tilde\rho: & G\ltimes_\theta\tilde G\to \Aff(A\times \tilde A)\\
&(g,h)\mapsto (\rho_g, \tilde \rho_{h}\circ\tilde b\circ \theta_g \circ \tilde b^{-1}).
\end{align*}

\begin{con}
The semi-direct product $\rho\ltimes_\theta\tilde\rho$ is an affine action of $G\ltimes_\theta\tilde G$ on $A\times \tilde A$. 
\end{con}

\begin{proof}
First notice that \eqref{E:asemi} guarantees that the mapping 
\[
(g,h)\mapsto (\pi_{g},\pi_{h}\tilde b\circ \theta_{g}\circ \tilde b^{-1})
\]
is a group homomorphism from $G\ltimes_\theta \tilde G$ to $\aut(A\times \tilde A)$. 
The tedious verification is left to the reader.

We now show the following identity: 
\begin{align}
\label{E:theta}
\theta_g(\tilde b^{-1}\tilde\rho_{h}\tilde b)=(\tilde b^{-1}\tilde \rho_{\theta_g(h)}\tilde b) \theta_g \qforal g\in G\text{ and } h\in \tilde G. 
\end{align}
In fact, one has 
\begin{align*}
&\ \theta_g(h_1h_2)=\theta_g(h_1)\theta_g(h_2)\qforal g\in G, h_1,h_2\in \tilde G\\
\Rightarrow&\  \tilde b(\theta_g(h_1h_2))=\tilde b(\theta_g(h_1)\theta_g(h_2))\\
\Rightarrow&\  \tilde b(\theta_g(\tilde b^{-1}(\tilde b(h_1)\tilde\pi_{h_1}(\tilde b(h_2))))=\tilde b(\theta_g(h_1))\tilde \pi_{\theta_g(h_1)}(\tilde b \theta_g(h_2))
\ (\text{as }\tilde b \text{ is a 1-cocycle})\\
\Rightarrow&\  \tilde b\theta_g\tilde b^{-1}\tilde\rho_{h_1}(\tilde b(h_2))=\tilde \rho_{\theta_g(h_1)}(\tilde b(\theta_g(h_2)))
\ (\text{by the definition of }\tilde\rho)\\
\Rightarrow&\  \theta_g(\tilde b^{-1}\tilde\rho_{h_1}\tilde b)=(\tilde b^{-1}\tilde \rho_{\theta_g(h_1)}\tilde b) \theta_g.
\end{align*}

Set $\Gamma:= \rho\ltimes_\theta\tilde\rho$. In order to show that $\Gamma$ is an affine action, it suffices to check that 
\[
\Gamma_{(g,h)(g',h')}=\Gamma_{(g,h)}\Gamma_{(g',h')}\qforal g,g'\in G, h,h'\in\tilde G.
\]
For this, let $y\in G$ and $t\in \tilde G$. We have 
\begin{align*} 
\Gamma_{(g,h)(g',h')}(y,t)
&=\Gamma_{(gg',h\theta_g(h')}(y,t)\\
&=\big(\rho_{gg'}(y),\tilde\rho_{h\theta_g(h')}\,\tilde b\,\theta_{gg'} \tilde b^{-1}(t)\big)\\
&=\big(\rho_{gg'}(y),\tilde\rho_{h\theta_g(h')}\,\tilde b\,\theta_{g}\theta_{g'}\tilde b^{-1}(t)\big)\\
&=\big(\rho_{gg'}(y),\tilde\rho_{h} \tilde b \tilde b^{-1}\tilde\rho_{\theta_g(h')}\,\tilde b\,\theta_{g} \theta_{g'}\tilde b^{-1}(t)\big)\\
&=\big(\rho_{gg'}(y),\tilde\rho_{h} \tilde b \theta_g\tilde b^{-1}\tilde\rho_{h'} \,\tilde b\, \theta_{g'}\tilde b^{-1}(t)\big)\ (\text{by }\eqref{E:theta}),
\end{align*}
and 
\begin{align*} 
\Gamma_{(g,h)}\Gamma_{(g',h')}(y,t)
&=\Gamma_{(g,h)}\big(\rho_{g'}(y),\tilde \rho_{h'}\tilde b\theta_{g'}\tilde b^{-1}(t)\big)\\
&=\big(\rho_g\rho_{g'}(y), \tilde \rho_h\tilde b\theta_g\tilde b^{-1}(\tilde \rho_{h'}\tilde b\theta_{g'}\tilde b^{-1}(t))\big).
\end{align*}
We are done. 
\end{proof}

When $\theta$ is the trivial action, then the condition \eqref{E:asemi} is redundant and the corresponding affine action is just the direct product of $\rho$ and $\tilde\rho$. 

An application of the above semi-direct product construction is given in the appendix.


\subsection{Constructing YBE solutions} 
\label{SS:constructing}

Let us first recall the following result. 

\begin{thm}\cite{LYZ00}
\label{T:gYBE}
Let $G$ be a group. Then following two groups of data are equivalent: 
\begin{enumerate}
\item 
There is a pair of left-right actions $(\alpha,\beta)$ of the group $G$ on $G$, which is compatible (i.e., $gh=\alpha_g(h)\beta_h(g)$ for all $g,h$ in $G$). 

\item There is a regular affine action $\rho=(\pi,b)$ of $G$ on some group $A$. 
\end{enumerate}
\end{thm}

 \begin{proof}
This is proved in \cite{LYZ00}. Since the idea of the proof will be useful later, we sketch it below. 
 
 (i)$\Rightarrow$(ii):  
 Let $A:=G$ as sets but the multiplication $\odot$ on $A$ is given by 
\[
 g\odot h=g\alpha_{g^{-1}}(h)\qforal g,h\in G, 
 \]
 namely,
  \[
 g h=g\odot\alpha_g(h)\qforal g,h\in G. 
 \]
 This implies that the identity mapping $\id$ is a (bijective) 1-cocycle with respect to $\alpha$. 
 
 (ii)$\Rightarrow$(i):
Set 
\[
\alpha_g(h):=b^{-1}\circ\pi_g\circ b(h)\quad\text{and}\quad \beta_h(g):=\alpha_g(h)^{-1}gh 
\]
for all $g,h \in G$. 
\end{proof}

\begin{rem}
\label{R: extension}
(i) Let $G$ and $A$ be groups. Given a regular affine action $\rho$ of $G$ on $A$, by Theorem \ref{T:gYBE} and  \cite[Corollary 1]{LYZ00}, one obtains a YBE solution on $G$ given by 
$R(g,h)=(\alpha_g(h),\beta_h(g))$ for all $g,h\in G$. 

(ii) Let $R$ be a YBE solution on $X$, and $\rho^X$ be its associated regular affine action of $G_X$ on $A_X$ (see Proposition \ref{P:YBEaff}). From (i) above, there is a YBE solution $\bar{R}$ on $G_X$. From its construction,  one can see that this is nothing but the universal extension of $R$ mentioned in \cite[Theorem 4.1]{LYZ00}. 

\end{rem}

\begin{rem}
This remark shows that there is a natural generalization of the relation $\beta_h(g)=\alpha^{-1}_{\alpha_g(h)}(g)$ holding for symmetric YBE solutions (cf. Lemma \ref{L:basic}). 

Let us return to the proof of (i)$\Rightarrow$(ii) in Theorem \ref{T:gYBE}. 
The property of $b:=\id$ being a 1-cocycle with respect to $\alpha$ gives 
\[
g\odot h=g\alpha_{g^{-1}}(h)=g\alpha_g^{-1}(h)\qforal g,h\in G,
\]
which implies 
\[
gh=g\odot\alpha_g(h)\qforal g,h\in G. 
\]
In particular\footnote{To distinguish, we write $\bar g$ as the inverse of $g$ in $A$, while $g^{-1}$ as the inverse of $g$ in $G$ as usual.}, 
\[
\bar g=\alpha_g(g^{-1})\qforal g\in G. 
\]

If $(\alpha,\beta)$ is a compatible pair, then we claim
\[
\beta_h(g)=\alpha_{{\alpha_g(h)}^{-1}}\circ\Ad_{\ol{\alpha_g(h)}}(g),
\]
where $\Ad_{\ol{\alpha_g(h)}}$ acts on $A$. 

Indeed, since $(\alpha,\beta)$ is a compatible pair, one has 
\begin{align*}
&\  g\odot h=g\alpha_{g^{-1}}(h)=h\beta_{\alpha_{g^{-1}}(h)}(g)\\
\Rightarrow&\ g\odot \alpha_g(h)=\alpha_g(h)\beta_h(g)\\
\Rightarrow&\ \beta_h(g)=\alpha_g(h)^{-1}(g\odot \alpha_g(h))\\
&\hskip 1cm =\alpha_g(h)^{-1}\odot \alpha_{\alpha_g(h)^{-1}}\big(g\odot \alpha_g(h)\big)\\
&\hskip 1cm =\alpha_{\alpha_g(h)^{-1}}\left[\alpha_{\alpha_g(h)}(\alpha_g(h)^{-1})\odot g\odot \alpha_g(h)\right]\ (\text{as }\alpha_g\in\aut(A))\\
&\hskip 1cm =\alpha_{\alpha_g(h)^{-1}}\left[\ol{\alpha_g(h)}\odot g\odot \alpha_g(h)\right]\\
&\hskip 1cm =\alpha_{{\alpha_g(h)}^{-1}}\left[\ol{\alpha_g(h)}\odot g\odot \alpha_g(h)\right]\\
&\hskip 1cm =\alpha_{{\alpha_g(h)}^{-1}}\circ\Ad_{\ol{\alpha_g(h)}}(g).
\end{align*}
This proves our claim. 

In particular, if the YBE solution $R$ on $G$ determined by $(\alpha,\beta)$ is symmetric, then $A$ is abelian \cite{LYZ00}. So in this case
 $\Ad_g$ ($g\in A$) is nothing but the identity mapping on $A$. 
 \end{rem}

Making use of Theorem \ref{T:gYBE}, Remark \ref{R: extension} and the constructions of affine actions in Subsection \ref{SS:aff}, we get two constructions of YBE solutions on groups. 

\subsection*{Lifting revisited} Let ${R_X}$ be a YBE solution.  
In the lifting construction on affine actions, let $G=G_X$, $A=A_X$, and $\rho$ be the affine action associated to ${R_X}$. 
Then \eqref{E:mul} becomes 
\[
x\cdot y=x\sigma_{\theta(x)}(y)\qforal x,y\in H. 
\]
In this case, 
$\tilde\rho_x(y)=x\cdot y$ is a regular affine action, and so it yields a YBE solution on $(H,\cdot)$.  

\subsection*{Semi-direct product revisited} 

 Let ${R_X}$ and ${R_Y}$ be two YBE solutions. Let $G=G_X$, $A=A_X$, $\tilde G=G_Y$, $\tilde A=A_Y$
 in the semi-direct product construction on affine actions.
Suppose $\theta$ is an action of $G_X$ on $G_Y$ satisfying \eqref{E:asemi}. In this case, 
\[
\rho^X\ltimes_\theta\rho^Y(g,h)=(\rho^X_g, \rho_h^Y\circ\theta_g),
\]
and $\rho^X\ltimes_\theta\rho^Y$ is also regular. 
It follows that $\rho^X\ltimes_\theta\rho^Y$ determines a YBE solution, say $\bar R$, on $G_X\ltimes_\theta G_Y$. 
Notice that if $\theta$ is the trivial action, then $\bar R$ is nothing but the trivial extension of ${R_X}$ and ${R_Y}$ in the sense of \cite{ESS}
(also cf. \cite[2.2 2$^\circ$]{Yan16}).


\section{Classifying solutions of the Yang-Baxter equation via their associated affine actions}

\label{S:classification}

In this section, we state and prove our main results in this paper. We classify all injective YBE solutions 
in terms of their associated regular affine actions (Theorem \ref{T:giso}). Furthermore, a connection with C*-dynamical systems is 
obtained: All injective YBE solutions can also be classified via their associated C*-dynamical systems (Theorem \ref{T:ds}). 

Let $R_X$ be a YBE solution. Denote by $\iota_G$ and $\iota_A$ the natural mappings from $X$ into $G_X$ and $A_X$, respectively. 

\begin{defn}
\label{D:inj}
If $\iota_G$ is injective, then ${R_X}$ is said to be injective. 
\end{defn}

It is known from \cite{Sol00} that $\iota_G$ is injective if and only if so is $\iota_A$. Also,
every symmetric YBE solution is injective. 

Let ${R_X}$ and ${R_Y}$ be two YBE solutions. 
Recall that a mapping $h: X\to Y$ is a \textit{YB-homomorphism between ${R_X}$ and ${R_Y}$}, if 
${R_Y}(h\times h)=(h\times h){R_X}$. This amounts to saying that 
\begin{align}
\label{E:alphainv}
\alpha^Y_{h(x)}(h(y))=h(\alpha^X_x(y))\quad \text{and}\quad \beta^Y_{h(x)}(h(y))=h(\beta^X_x(y))
\end{align}
for all $x,y\in X$. 
In this case, we also say that $R_X$ is \textit{homomorphic to $R_Y$ via $h$}. 
Of course, if $h$ is bijective, then $R_X$ and $R_Y$ are called \textit{isomorphic}. 

If ${R_X}$ and ${R_Y}$ are symmetric, then only one of the two identities in \eqref{E:alphainv} suffices.

\begin{prop}
\label{P:symcon}
Let ${R_X}$ and ${R_Y}$ be two arbitrary YBE solutions. If $R_X$ is homomorphic to $R_Y$ via $h$, 
then $h$ induces group homomorphisms $h_G: G_X\to G_Y$ and $h_A: A_X\to A_Y$, 
such that $h_A$ is $G_X$-equivariant relative to $(\rho^X,\rho^Y \circ h_G)$.

If $h$ is furthermore bijective, then $\rho^X$ and $\rho^Y \circ h_G$ are conjugate. 
\end{prop}

\begin{proof}
For convenience, let ${R_X}(x_1,x_2)=(\alpha^X_{x_1}(x_2),\beta^X_{x_2}(x_1))$ for all $x_1,x_2\in X$, and 
${R_Y}(y_1,y_2)=(\alpha^Y_{y_1}(y_2),\beta^Y_{y_2}(y_1))$ for all $y_1,y_2\in Y$. 

Notice that since $h:X\to Y$ is a YB-homomorphism between $R_X$ and $R_Y$, it is easy to check that $h$ can be extended to a
group homomorphism, say $h_G$, from $G_X$ to $G_Y$. 
Indeed, it follows from \eqref{E:alphainv} and the definition of $G_Y$ that 
\[
h(\alpha^X_x(y))h(\beta^X_y(x))=\alpha^Y_{h(x)}(h(y))\beta^Y_{h(y)}(h(x))=h(x)h(y) 
\]
for all $x,y\in X$. Obviously, $\rho^Y\circ h_G$ is an affine action of $G_X$ on $A_Y$. 

Similarly, one can extend $h$ to a group homomorphism, say $h_A$, from $A_X$ to $A_Y$. In fact,
repeatedly using \eqref{E:alphainv} yields
\begin{align*}
&\ h(\alpha^X_x(y))\bullet h(\alpha^X_{\alpha^X_x(y)}(\beta^X_y(x)))\\
=&\ \alpha^Y_{h(x)}(h(y))\bullet \alpha^Y_{h(\alpha_x(y))}(h(\beta^X_y(x)))\\
=&\ \alpha^Y_{h(x)}(h(y))\bullet \alpha^Y_{\alpha^Y_{h(x)}(h(y))}(\beta^Y_{h(y)}(h(x)))
\end{align*}
for all $x,y\in X$. 
But the definition of $A_Y$ gives 
\[
h(x)\bullet \alpha^Y_{h(x)}(h(y))=\alpha^Y_{h(x)}(h(y))\bullet \alpha^Y_{\alpha^Y_{h(x)}(h(y))}(\beta^Y_{h(y)}(h(x))),
\]
and so 
\begin{align*}
h(x)\bullet h(\alpha^X_x(y))
&=h(x)\bullet \alpha^Y_{h(x)}(h(y))\\
&=h(\alpha^X_x(y))\bullet h(\alpha^X_{\alpha^X_x(y)}(\beta^X_y(x)))\qforal x,y\in X.
\end{align*}

In what follows, we show that $h_A$ is $G_X$-equivariant relative to $\rho^X$ and $\rho^Y\circ h_G$. By Remark 
\ref{R:equ} it is equivalent to show 
 \begin{align}
 \label{E:halpha}
  \alpha^Y_{h_G(g)}(h_A(a))&=h_A(\alpha^X_g(a))\qforal g\in G_X, a\in A_X,\\
 \label{E:hb}
 h_A(b^X(g))&=b^Y(h_G(g))\qforal g\in G_X. 
 \end{align}
Applying \eqref{E:alphainv} and Proposition \ref{P:inv}, one has that
\[
\alpha^Y_{h_G(g)}(h_A(x))=h_A(\alpha^X_g(x))\qforal g\in G_X, x\in X.
\]
Now from this identity and Proposition \ref{P:inv}, one can easily verify \eqref{E:halpha}.  
  
For \eqref{E:hb}, first notice that it is true when $g\in X$, as both sides are equal to $h(g)$. 
Then the general case follows from \eqref{E:halpha} and the definition of 1-cocycles.
 
 The last assertion of the proposition is clear. 
\end{proof}

The following theorem generalizes the case of symmetric YBE solutions (cf., e.g., \cite{ESS}). 

\begin{thm}
\label{T:giso}
Let ${R_X}$ and ${R_Y}$ be two injective YBE solutions. Then they are isomorphic, if and only if there is a group isomorphism $\phi:G_X\to G_Y$
such that $\phi(X)=Y$, and  $\rho^X$ and $\rho^Y\circ\phi$ are conjugate. 
\end{thm}

\begin{proof}
\underline{``Only if" part:} Let $h:X\to Y$ be a YB-isomorphism between ${R_X}$ and ${R_Y}$. 
Keep the same notation used in the proof of Proposition \ref{P:symcon}.
Then $\phi:=h_G$ has all desired properties, and furthermore $\rho^X$ and $\rho^Y\circ \phi$ are conjugate via $h_A$. 

\underline{``If" part:} 
As before, write $\rho^X=(\alpha^X, b^X)$ and $\rho^Y=(\alpha^Y, b^Y)$.
Let $h: A_X\to A_Y$ be a $G_X$-equivariant mapping relative to $(\rho^X, \rho^Y \circ\phi)$.  Then by Remark \ref{R:equ}  we have 
\begin{align}
\label{E:alpha}
h \circ\alpha^X_g&=\alpha^Y_{\phi(g)}\circ h,\\
\label{E:b}
b^Y \circ \phi(g) &=h\circ b^X(g)
\end{align}
for all $g\in G_X$. 

On the other hand, it follows from the proof of Theorem \ref{T:gYBE}  and Remark \ref{R: extension} that $\rho^X$ and $\rho^Y$ induce YBE solutions $\bar{R}_X$ and 
$\bar{R}_Y$ on $G_X$ and $G_Y$, respectively. Actually, 
\begin{align*}
\bar{R}_X(g_1,g_2)&=(\tilde\alpha^X_{g_1}(g_2), \tilde\beta^X_{g_2}(g_1))\qforal g_1,g_2\in G_X, \\
\bar{R}_Y(g_1',g_2')&=(\tilde\alpha^Y_{g_1'}(g_2'), \tilde\beta^Y_{g_2'}(g_1'))\qforal g_1',g_2'\in G_Y,
\end{align*}
where 
\[
\tilde\alpha^X_{g_1}:=(b^X)^{-1}\alpha^X_{g_1}b^X, \quad\tilde\beta^X_{g_2}(g_1):=\tilde\alpha^X_{g_1}(g_2)^{-1}g_1g_2, 
\]
and similarly for $\tilde\alpha^Y$, $\tilde\beta^Y$.  

From \eqref{E:b} one has 
\begin{align}
\label{E:tildeh}
\phi=(b^Y)^{-1}\circ h\circ b^X. 
\end{align}
We claim that $\phi$ is actually a YB-isomorphism between $\bar{R}_X$ and $\bar{R}_Y$. 
To this end, we must show that the two identities in \eqref{E:alphainv} hold true. 

$\blacktriangleright$ Firstly, we check 
\[
\phi\circ\tilde\alpha^X_g=\tilde\alpha^Y_{\phi(g)}\circ\phi\qforal g\in G_X.
\]
But this follows from \eqref{E:alpha}, the definitions of $\tilde\alpha^X$ and $\tilde\alpha^Y$: 
\begin{align*}
\eqref{E:alpha}\Rightarrow &\ hb^X((b^X)^{-1} \alpha^X_gb^X)=b^Y((b^Y)^{-1}\alpha^Y_{\phi(g)}b^Y)((b^Y)^{-1} hb^X)\\
\Rightarrow&\ [(b^Y)^{-1}hb^X][(b^X)^{-1} \alpha^X_gb^X]=[(b^Y)^{-1}\alpha^Y_{\phi(g)}b^Y][(b^Y)^{-1} hb^X]\\
\Rightarrow&\ \phi \circ \tilde \alpha^X_g=\tilde\alpha^Y_{\phi(g)}\circ\phi.
\end{align*}

$\blacktriangleright$ Secondly, we verify that 
\[
\phi\circ\tilde\beta^X_g=\tilde\beta^Y_{\phi(g)}\circ\phi\qforal g\in G_X.
\]

Since $b^X$ is a 1-cocycle with respect to $\alpha^X$ in coefficient $A_X$, one has that for all $g_1,g_2\in G_X$
\begin{align}
\nonumber
&\ b^X_{g_1g_2}=b^X_{g_1}\bullet\alpha^X_{g_1}(b^X_{g_2})\\
\nonumber
\Rightarrow &\ g_1(b^X)^{-1}(\alpha^X_{g_1^{-1}}(b^X_{g_2}))=(b^X)^{-1}(b^X_{g_1}\bullet b^X_{g_2})=g_1\odot g_2\\
\nonumber
\Rightarrow &\ g_1\tilde\alpha^X_{g_1^{-1}}(g_2)=g_1\odot g_2\\
\nonumber
\Rightarrow &\ g_1^{-1}\tilde\alpha^X_{g_1}(g_2)=g_1^{-1}\odot g_2\\
\label{E:alpha-1}
\Rightarrow &\ \tilde\alpha^X_{g_1}(g_2)^{-1}=(g_1^{-1}\odot g_2)^{-1}g_1^{-1}.
\end{align}
Similarly, 
\[
\tilde\alpha^Y_{g_1'}(g_2')^{-1}=(g_1'^{-1}\odot g_2')^{-1}g_1'^{-1}\qforal g_1',g_2'\in G_Y. 
\tag{12'}
\]

Now define a new multiplication $\odot$ on $G_X$ by
\[
g_1\odot g_2=(b^{X})^{-1}(b^X_{g_1}\bullet b^X_{g_2}) \qforal g_1,g_2\in G_X,
\]
and similarly on $G_Y$. Then it is easy to check that $(G_X,\odot)$ and $(G_Y,\odot)$ are groups.  
In what follows, we claim that $\phi$ is also a group homomorphism from $(G_X, \odot)$ to $(G_Y, \odot)$. 
As a matter of fact, for all $g_1$, $g_2\in G_X$ one has 
\begin{align*}
\phi(g_1\odot g_2)
&=\phi\circ (b^X)^{-1}(b^X_{g_1}\bullet b^X_{g_2})\\
&=(b^Y)^{-1}\circ h(b^X_{g_1}\bullet b^X_{g_2})\ (\text{by }\eqref{E:b})\\
&=(b^Y)^{-1}(h(b^X_{g_1})\bullet h(b^X_{g_2}))\ (\text{as }h:A_X\to A_Y\text{ is a homomorphism})\\
&=(b^Y)^{-1}(b^Y_{\phi(g_1)}\bullet b^Y_{\phi(g_2)})\ (\text{by }\eqref{E:b})\\
&=\phi(g_1)\odot  \phi(g_2) \ (\text{by the definition of }\odot\text{ in }G_Y).
\end{align*}

We now have
\begin{align*}
  \phi\circ \tilde\beta^X_{g_2}(g_1)
&=\phi(\tilde\alpha^X_{g_1}(g_2)^{-1}g_1g_2)\ (\text{by the definition of }\tilde\beta^X)\\
&=\phi((g_1^{-1}\odot g_2)^{-1}g_1^{-1}g_1g_2)\ (\text{by }\eqref{E:alpha-1})\\
&=\phi(g_1^{-1}\odot g_2)^{-1}\phi(g_2)\ (\text{as }\phi \text{ is a homomorphism from }G_X\text{ to }G_Y)\\
&=(\phi(g_1^{-1})\odot \phi(g_2))^{-1}\phi(g_2)\ (\text{by the above claim})\\
&=(\phi(g_1)^{-1}\odot \phi(g_2))^{-1}\phi(g_2)\ (\text{as }\phi \text{ is a homomorphism from }G_X\text{ to }G_Y)\\
&=(\phi(g_1)^{-1}\odot \phi(g_2))^{-1}\phi(g_1)^{-1}\phi(g_1)\phi(g_2)\\
&=\tilde\alpha^X_{\phi(g_1)}(\phi(g_2))^{-1}\phi(g_1)\phi(g_2)\ (\text{by (12')})\\
&=\tilde\beta^X_{\phi(g_2)}(\phi(g_1))\ (\text{by the definition of }\tilde\beta^Y)
\end{align*}
for all $g_1,g_2\in G_X$. 

Therefore, $\phi$ is a YB-isomorphism between $\bar {R}_X$ and $\bar {R}_Y$. 

Recall from Remark \ref{R: extension} that $\bar{R}_X$ is an extension of $R_X$ from $X$ to $G_X$ and similarly for $\bar R_Y$. 
Since $R_X$ and $R_Y$ are injective and $\phi(X)=Y$, the restriction $\phi|_X$ yields a YB-isomorphism between ${R_X}$ and ${R_Y}$. 
\end{proof}

We are now ready to provide a characterization when the extensions $\bar{R}_X$ and $\bar{R}_Y$ are isomorphic. 

\begin{thm}
\label{T:iso}
Let ${R_X}$ and ${R_Y}$ be two arbitrary YBE solutions.  Then the extensions $\bar R_X$ and $\bar R_Y$ on $G_X$ and $G_Y$ are YB-isomorphic,
if and only if there is a group isomorphism $\phi:G_X\to G_Y$ such that $\rho^X$ and $\rho^Y\circ\phi$ are conjugate. 
\end{thm}

\begin{proof}
$(\Leftarrow)$: It directly follows from the proof of ``If" part of Theorem \ref{T:giso}. 

($\Rightarrow$): Let $h:G_X\to G_Y$ be a YB-isomorphism between $\bar R_X$ and $\bar R_Y$. Now consider $h|_{\iota_G(X)}$. Then completely similar to the proof of Proposition \ref{P:symcon}, $h|_{\iota_G(X)}$ can be extended to an isomorphism $\phi$ from $G_X$ to $G_Y$, such that $\rho^X$ and $\rho^Y\circ\phi$ are conjugate. 
\end{proof}

\smallskip
In the rest of this section, we provide a connection with C*-dynamical systems. For any group $G$, by $\ca(G)$ we mean the group C*-algebra of $G$. Since all groups here 
are assumed to be discrete, $\ca(G)$ is unital. Furthermore, $G$ can be canonically embedded to $\ca(G)$ as its unitary generators. 
For the background on C*-dynamical systems which is needed below, refer to \cite{Bla06}.

\begin{prop}
\label{P:pi}

(i) A YBE solution ${R_X}$ determines an action $\pi^X$ of $G_X$ on
$M_2(\ca(A_X))$ such that 
\[
\pi_g^X(\diag(x,y))=\diag(\gamma_g(x), \zeta_g(y)) \qforal g\in G_X, x,y\in\ca(A_X),
\] 
where $\gamma$ and $\zeta$ are representations of $G_X$ on $\ca(A_X)$. 

(ii) If $h$ is a YB-homomorphism between ${R_X}$ and ${R_Y}$, then there are group homomorphisms $h_G : G_X\to G_Y$ and $h_A:A_X\to A_Y$ such that 
the inflation $h_A^{(2)}$ is $G_X$-equivariant relative to $(\pi^X,\pi^Y \circ h_G)$. 
\end{prop}

\begin{proof}
(i) Let $\pi^X$ be defined as 
\begin{align*}
\pi^X_g\left(\begin{matrix}a_1&a_2\\a_3&a_4\end{matrix}\right)
&=\left(\begin{matrix}\alpha^X_g(a_1)&\alpha^X_g(a_2)(b^X_g)^*\\ b^X_g\alpha^X_g(a_3)&b^X_g\alpha^X_g(a_4)(b^X_g)^*\end{matrix}\right)\\
&=\left(\begin{matrix}1&0\\ 0& b^X_g\end{matrix}\right) (\alpha_g^X)^{(2)}\left(\begin{matrix}a_1&a_2\\ a_3&a_4\end{matrix}\right)
   \left(\begin{matrix}1&0\\ 0& (b^X_g)^*\end{matrix}\right) 
\end{align*}
 for all $g\in G_X$ and $a_1,a_2,a_3,a_4\in A_X$. Then one can use the properties of $\alpha$ and $b$ to easily check that $\pi^X$
 is an action of $G_X$ on the matrix C*-algebra $M_2(\ca(A_X))$. Also $\gamma_g(\cdot)=\alpha_g^X(\cdot)$ and $\zeta_g(\cdot)=b^X_g\alpha^X_g(\cdot)(b^X_g)^*$
 are two representations of $G_X$ on $\ca(A_X)$. 

(ii) Since $h$ is a YB-isomorphism between $R_X$ and $R_Y$, as in the proof of Proposition \ref{P:symcon},  it induces group 
homomorphisms $h_G: G_X\to G_Y$ and $h_A:A_X\to A_Y$ satisfying \eqref{E:halpha} and \eqref{E:hb}. Then we extend $h_A$ to a C*-homomorphism, 
still denoted by $h_A$, from $\ca(A_X)$ to $\ca(A_Y)$. 
Furthermore, its inflation $h_A^{(2)}: M_2(\ca(A_X))\to M_2(\ca(A_Y))$ gives a $G_X$-equivariant mapping relative to $\pi^X$ and $\pi^Y\circ h_G$.  
In fact, a simple calculation gives 
\[
h_A^{(2)} \circ \pi^X_g\left(\begin{matrix}a_1&a_2\\a_3&a_4\end{matrix}\right)
=\left(\begin{matrix}
h_A(\alpha^X_g(a_1))&h_A(\alpha^X_g(a_2))h_A((b^X_g)^*)\\ h_A(b^X_g)h_A(\alpha^X_g(a_3))&h_A(b_g)h_A(\alpha^X_g(a_4))h_A((b^X_g)^*)
\end{matrix}\right)
\]
and 
\[
\pi^Y_{h_G(g)}\circ h_A^{(2)}\left(\begin{matrix}a_1&a_2\\a_3&a_4\end{matrix}\right)
=\left(\begin{matrix}
\alpha_{h_G(g)}^Y(h_A(a_1))&\alpha_{h_G(g)}^Y(h_A(a_2))(b^Y_{h_G(g)})^*\\ b^Y_{h_G(g)}\alpha_{h_G(g)}^Y(h_A(a_3))&b^Y_{h_G(g)}\alpha_{h_G(g)}(a_4) (b^Y_{h_G(g)})^*
\end{matrix}\right).
\]
Then apply \eqref{E:halpha} and \eqref{E:hb} to obtain the right hand sides equal.
\end{proof}

As a consequence of Proposition \ref{P:pi}, from the associated regular affine action $\rho^X$ of a given YBE solution $R_X$, one obtains a C*-dynamical system $(G_X, M_2(\ca(A_X)),\pi^X)$.

\begin{thm}
\label{T:ds}
Two injective YBE solutions ${R_X}$ and ${R_Y}$ are isomorphic, if and only if there is a group isomorphism $\phi: G_X\to G_Y$ mapping $X$ onto $Y$ such that $(G_X, M_2(\ca(A_X)),\pi^X)$ and $(G_X, M_2(\ca(A_Y)),\pi^Y\circ \phi)$ are conjugate. 
\end{thm}

\begin{proof}
($\Rightarrow$): Keep the same notation as in the proof of Proposition \ref{P:pi}. If $h: X\to Y$ is a YB-isomorphism between ${R_X}$ and ${R_Y}$, then $\phi:=h_G$, and $\pi^X$ and $\pi^Y\circ \phi$ are equivalent via $h_A^{(2)}$. 

($\Leftarrow$): Let $\fh:M_2(\ca(A_X))\to M_2(\ca(A_Y))$ be an intertwining homomorphism between  $\pi^X$ and $\pi^Y\circ \phi$. Let us write 
$\fh=\left(\begin{matrix} h_{11}&h_{12}\\ h_{21}& h_{22}\end{matrix}\right)$. 
Then $\fh\circ\pi_g^X\left(\begin{matrix} a&0\\0&0\end{matrix}\right)= \pi^X_{\phi(g)}\circ\fh\left(\begin{matrix} a&0\\0&0\end{matrix}\right)$ yields 
\[
h_{11}\alpha^X_g(a)=\alpha^X_{\phi(g)}h_{11}(a)\qforal a\in A_X.
\]
Also $\fh\circ\pi_g^X\left(\begin{matrix} 0&I\\0&0\end{matrix}\right)= \pi^Y_{\phi(g)}\circ\fh\left(\begin{matrix} 0&I\\ 0&0\end{matrix}\right)$ yields 
$
(h_{11}(b^X_g))^*=(b^Y_{\phi(g)})^*,
$
which implies 
\[
h_{11}(b^X_g)=b^Y_{\phi(g)}.
\]
The above two identities give \eqref{E:alpha} and \eqref{E:b}. Then applying the proof of ``If" part of Theorem \ref{T:giso}  ends the proof. 
\end{proof}

\appendix\section{A Commutation Relation for Semi-Direct Products}

\label{S:cycle}

In this appendix, we prove a commutation relation for semi-direct products of YBE solutions derived from cycle sets, which might be useful in the future studies. 
We further describe a connection between the structure group of the semi-direct product of two YBE solutions and the semi-direct product of their structure
groups.

\begin{defn}
\label{D:cycle}
A non-empty set $X$ with a binary relation $\cdot$ is called a cycle set, if 
\[
(x\cdot y)\cdot (x\cdot z)=(y\cdot x)\cdot (y\cdot z)\qforal x,y,z\in X. 
\]

A cycle set $X$ is said to be non-degenerate if $x\mapsto x\cdot x$ is bijective. 
\end{defn}

The main motivation to study cycle sets is the following theorem due to Rump (\cite{Rum05}): There is a one-to-one correspondence between the set of symmetricYBE solutions and the set of non-degenerate cycle sets. 
In fact, let $X(,\cdot)$ be a non-degenerate cycle set. If we let $\ell_x(y)=x\cdot y$, then 
\[
R(x,y)=\left(\ell_{\ell^{-1}_y(x)}(y), \ell_y^{-1}(x)\right)
\] 
is a symmetric YBE solution on $X$. 
Conversely, given a symmetric YBE solution $R(x,y)=(\alpha_x(y),\beta_y(x))$ on $X$, let $x\cdot y=\beta_x^{-1}(y)$. 
Then $(X,\cdot)$ is a non-degenerate cycle set.

\smallskip

The following small lemma turns out very handy. 

\begin{lem}
\label{L:xyx}
Keep the above notation. Then 
\begin{align*}
G_{R_X}={}_{{\rm gp}}\left\langle X; (y\cdot x)y=(x\cdot y)x\qforal x,y\in X\right\rangle.
\end{align*}
\end{lem}

\begin{proof}
By Lemma \ref{L:basic}, $\alpha_x(y)=\beta^{-1}_{\beta_y(x)}(y)$ for all $x,y \in X$. Hence 
\begin{align*}
R(x,y)=(\beta^{-1}_{\beta_y(x)}(y),\beta_y(x))
\Leftrightarrow &\ R(\beta_y^{-1}(x),y)=(\beta^{-1}_x(y),x)\\
\Leftrightarrow &\ R(y\cdot x,y)=(x\cdot y,x)
\end{align*}
for all $x,y\in X$. 
\end{proof}

Let us now recall Rump's semi-direct product of cycle sets below. 

\begin{defn}
\label{D:rump}
Let $X$ and $S$ be two finite cycle sets, and $\pi$ be an action of $X$ on $S$. That is,
$\pi: X\times S\to S, \ (x,s)\mapsto \pi_x(s)$, satisfies 
\begin{enumerate}
\item
$\pi_x(s\cdot t)=\pi_x(s)\cdot \pi_x(t)$ for every $x\in X$ and for all $s,t\in S$;

\item $\pi_{y\cdot x}\pi_y(s)=\pi_{x\cdot y}\pi_x(s)$ for all $x,y\in X$ and $s\in S$;

\item $\pi_x\in \sym(S)$ for every $x\in X$. 
\end{enumerate}
\end{defn}

Set 
\begin{align*}
\gamma_{x,y}(s,t)=\pi_{x\cdot y}(s)\cdot \pi_{y\cdot x}(t).
\end{align*}
Now define  
\begin{align}
\label{E:sdp}
(x,s)\cdot (y,t):=(x\cdot y, \gamma_{x,y}(s,t)).
\end{align}
Then this gives a cycle structure on $X\times S$, which is denoted by $X\ltimes_{\pi} S$, called the \textit{semi-direct product
of $X$ and $S$ by $\pi$}. The symmetric YBE solution determined by $X\ltimes_\pi S$ is written as $R_{X\ltimes_\pi S}$. 

\begin{rem}
\label{R:(i)}
Notice that for Definition \ref{D:rump} (i) one has 
\begin{align*}
\pi_x(s\cdot t)=\pi_x(s)\cdot \pi_x(t)
&\Leftrightarrow  \pi_x(\beta_s^{-1}(t))=\beta_{\pi_x(s)}^{-1}(\pi_x(t))\\
&\Leftrightarrow  \beta_{\pi_x(s)}(\pi_x(t))=\pi_x(\beta_s(t))\\
&\Leftrightarrow  \pi_x(\alpha_s(t))=\alpha_{\pi_x(s)}(\pi_x(t))
\end{align*}
for all $s,t\in S$ and $x\in X$. 
In particular, this shows that, for every $x\in X$, $\pi_x$ is a YB-isomorphism between $R_S$ and itself.  
\end{rem}

In the sequel, let us write
\[
R_X(x,y)=(\alpha_x(y),\beta_y(x))\quad{and}\quad R_S(s,t)=(\tilde\alpha_s(t), \tilde\beta_t(s)). 
\]

\begin{cor}
\label{C:semi}
Let $X$ and $S$ be cycles sets, and $\pi$ be an action of $X$ on $S$. 
Then the YBE solution $R_{X\ltimes_\pi S}$ is explicitly given by 
the following formula 
\[
R_{X\ltimes_\pi S}\big((x,s),(y,t)\big)=\left(\big(\alpha_x(y), \tilde\alpha_s(\pi_x(t))\big),\big(\beta_y(x),\pi_{\alpha_x(y)}^{-1}(\tilde\beta_{\pi_x(t)}(s))\big)\right)
\]
for all $x,y\in X$ and $s,t\in S$. 
\end{cor}

\begin{proof}
First observe that 
\[
x\cdot y=\beta_x^{-1}(y)\Rightarrow x\cdot \beta_x(y)=y\Rightarrow y\cdot \beta_y(x)=x
\]
and 
\[
x\cdot y=\beta_x^{-1}(y)\Rightarrow \beta_y(x)\cdot y=\beta^{-1}_{\beta_y(x)}(y)=\alpha_x(y).
\]
The above identities will be frequently used in the sequel. 

Suppose that 
\[
R_{X\ltimes_\pi S}((x,s),(y,t))=\big(\alpha'_{(x,s)}(y,t), \beta'_{(y,t)}(x,s)\big)\qforal x,y\in X, s,t\in S. 
\]
Let
$
\beta'_{(y,t)}(x,s)=(z,p).
$
Then 
\begin{align*}
(x,s)
=\ell_{(y,t)}(z,p)=(y,t)\cdot (z,p)
=(y\cdot z, \pi_{y\cdot z}(t)\cdot \pi_{z\cdot y}(p)).
\end{align*}
So 
$
z=\beta_y(x)
$
and 
\begin{align*}
s=\pi_{y\cdot z}(t)\cdot \pi_{z\cdot y}(p)
&\Rightarrow\pi_{z\cdot y}(p)=\tilde\beta_{\pi_{y\cdot z}(t)}(s)
\Rightarrow p=\pi_{z\cdot y}^{-1}(\tilde\beta_{\pi_{y\cdot z}(t)}(s))\\
&\Rightarrow p=\pi_{\beta_y(x)\cdot y}^{-1}(\tilde\beta_{\pi_{y\cdot \beta_y(x)}(t)}(s))
\Rightarrow p=\pi_{\alpha_x(y)}^{-1}(\tilde\beta_{\pi_x(t)}(s)).
\end{align*}
Thus 
\[
\beta'_{(y,t)}(x,s)=\big(\beta_y(x),\pi_{\alpha_x(y)}^{-1}(\tilde\beta_{\pi_x(t)}(s))\big).
\]

Now 
\[
\alpha'_{(x,s)}(y,t)=\beta'^{-1}_{\beta'_{(y,t)}(x,s)}(y,t)=:(u,v). 
\]
Then 
\begin{align*}
(u,v)
&=\big(\beta_y(x),\pi_{\alpha_x(y)}^{-1}(\tilde\beta_{\pi_x(t)}(s))\big)\cdot (y,t)\\
&=\big(\beta_y(x)\cdot y, \pi_{\beta_y(x)\cdot y}(\pi^{-1}_{\alpha_x(y)}(\tilde\beta_{\pi_x(t)}(s))\cdot \pi_{y\cdot \beta_y(x)}(t))\big)\\
&=\big(\beta_y(x)\cdot y, \pi_{\alpha_x(y)}(\pi^{-1}_{\alpha_x(y)}(\tilde\beta_{\pi_x(t)}(s))\cdot \pi_{x}(t))\big)\\
&=\big(\beta_y(x)\cdot y, \tilde\beta_{\pi_x(t)}(s)\cdot \pi_{x}(t)\big)\\
&=\big(\beta_y(x)\cdot y, \tilde\alpha_s(\pi_{x}(t))\big).
\end{align*}
Therefore
\[
\alpha'_{(x,s)}(y,t)=\big(\beta_y(x)\cdot y, \tilde\alpha_s(\pi_{x}(t))\big).
\]
This ends the proof.
\end{proof}

\begin{lem}
\label{L:caction}
Let $X$ and $S$ be cycle sets, and $\pi$ be an action of $X$ on $S$. Then $\pi$ can be extended to an action of $G_{R_X}$ on $G_{R_S}$. 
\end{lem}

\begin{proof}
Notice that, in Definition \ref{D:rump},
(i) says that $\pi_x$ is a cycle morphism on $S$ for every $x\in X$, and 
(ii) says that $\pi_{y\cdot x}\pi_y=\pi_{x\cdot y}\pi_x$ for all $x, y\in X$. 
Thus Lemma \ref{L:xyx} and the latter imply that the action $\pi$ can be extended to an action 
$
\pi: G_{R_X}\curvearrowright S.
$

Applying Lemma \ref{L:xyx} to $G_{R_S}$, one has 
\begin{align*}
(s\cdot t) s=(t\cdot s)t\qforal s, t\in S.
\end{align*} 
Since $\pi_x(s), \pi_x(t)\in S$ by Definition \ref{D:rump} (iii), 
 replacing $s$ and $t$ by $\pi_x(s)$ and $\pi_x(t)$, respectively, in the identity obtained above gives 
\[
(\pi_x(s)\cdot \pi_x(t))\pi_x(s)= (\pi_x(t)\cdot \pi_x(s))\pi_x(t)\qforal x\in X, s,t\in S.
\]
This implies 
\[
\pi_x(s\cdot t)\pi_x(s)= \pi_x(t\cdot s)\pi_x(t)\qforal x\in X, s,t\in S
\] 
as $\pi_x$ is a cycle morphism on $S$. Therefore, by Lemma \ref{L:xyx}, $\pi$ can be extended to an action $\pi: G_{R_X}\curvearrowright G_{R_S}$. 
\end{proof}

Under the conditions of Lemma \ref{L:caction}, one can form the semi-direct product $G_{R_X}\ltimes_\pi G_{R_S}$, where 
\[
\pi: G_{R_X}\times G_{R_S}\to G_{R_S}, \ (x,s)\mapsto \pi_x(s).
\] 
It is also worth mentioning that the identity \eqref{E:asemi} automatically holds true for the action $\pi$ of $G_{R_X}$ on $G_{R_S}$ obtained in Lemma \ref{L:caction}. 
In fact, let $\theta=\pi$ and so it suffices to show that 
$\pi_x(\tilde b^{-1}\alpha_s \tilde b)=(\tilde b^{-1}\alpha_{\pi_x(s)}\tilde b)\pi_x$ for all $x\in X$ and $s\in S$.
But the restrictions $b$ and $\tilde b$ onto $X$ and $S$ are the identity mappings. 
Thus this amounts to $\pi_x(\alpha_s(t))=\alpha_{\pi_x(s)}(\pi_x(t))$. But this holds true by Remark \ref{R:(i)}. 

Therefore, one obtains a regular affine action $\rho^X\ltimes_\pi\rho^S$ of $G_{R_X}\ltimes_\pi G_{R_S}$ on $A_{R_X}\times A_{R_S}$ 
(cf. Subsection \ref{SS:aff}), and so a YBE solution $\hat R$ on $G_{R_X}\ltimes_\pi G_{R_S}$.
(cf. Subsection \ref{SS:constructing}). It is natural to write $\hat R|_{(X\times S)^2}$ as $R_X\ltimes_\pi R_S$, 
called the \textit{semi-direct product of $R_X$ and $R_S$ by $\pi$}. 
Then we obtain the following commutation relation:

\begin{prop}[\textbf{Commutation Relation for Semi-Direct Products}]
\label{P:RXRS}
Let $X$ and $S$ be cycle sets, and $\pi$ be an action of $X$ on $S$. Then 
\[
R_X\ltimes_\pi R_S=R_{X\ltimes_\pi S}.
\]  
\end{prop}

\begin{proof}
Assume that 
\[
\hat R((x,s),(y,t))=\big(\hat\alpha_{(x,s)}(y,t), \hat\beta_{(y,t)}(x,s)\big)\qforal x,y\in X, s,t\in S. 
\]
It follows from Subsection \ref{SS:aff} $2^\circ$ that 
\[
\hat \alpha_{(x,s)}(y,t)=(\alpha_x(y),\tilde\alpha_s(\pi_x(t)).
\]
Then an easy calculation yields
\begin{align*}
\hat\beta_{(y,t)}((x,s))
&=\hat\alpha_{(x,s)}(y,t)^{-1}(xy,s\pi_x(t))\\
&=(\alpha_x(y)^{-1}xy, \pi_{\alpha_x(y)^{-1}}(\tilde\alpha_s(\pi_x(t))^{-1}s\pi_x(t)))\\
&=(\beta_y(x), \pi_{\alpha_x(y)}^{-1}(\tilde\beta_{\pi_x(t)}(s)). 
\end{align*}
Therefore comparing the formula of $R_{X\ltimes_\pi S}$ given in Corollary \ref{C:semi} yields the desired commutation relation. 
\end{proof}

For the structure groups $G_{R_X}$, $G_{R_S}$ and $G_{R_X\ltimes_\pi R_S}$, we have the following: 

\begin{prop}
\label{P:semidp}
Keep the above notation. Then there is a group homomorphism 
\[
\Pi: G_{R_X\ltimes_\pi R_S}\to G_{R_X}\ltimes_\pi G_{R_S}.
\]
\end{prop}

\begin{proof}
By Proposition \ref{P:RXRS}, $G_{R_X\ltimes_\pi R_S}=G_{R_{X\ltimes_\pi S}}$.
Applying Lemma \ref{L:xyx} to $G_{R_{X\ltimes_\pi S}}$, we have the following relations
\[
((x,s)\cdot (y,t))(x,s)=((y,t)\cdot (x,s))(y,t)\qforal x,y\in X, \ s,t\in S. 
\]
From \eqref{E:sdp}, this is equivalent to 
\[
(x\cdot y, \gamma_{x,y}(s,t))(x,s)=(y\cdot x, \gamma_{y,x}(t,s))(y,t).
\]

Let $\Pi: X\ltimes_\pi S\to G_{R_X}\ltimes_\pi G_{R_S}$ be defined via 
\[
\Pi(x,s)=(x,s)\qforal x\in X, s\in S.
\] 
Simple calculations show that 
\begin{align*}
  \Pi(x\cdot y, \gamma_{x,y}(s,t))\, \Pi(x,s)
&=\big(x\cdot y, \pi_{x\cdot y}(s)\cdot \pi_{y\cdot x}(t)\big)\, (x,s)\\
&=\big((x\cdot y)x, (\pi_{x\cdot y}(s)\cdot \pi_{y\cdot x}(t))\pi_{x\cdot y}(s)\big) 
\end{align*}
and 
\begin{align*}
  \Pi(y\cdot x, \gamma_{y,x}(t,s))\, \Pi(y,t)
&=(y\cdot x, \pi_{y\cdot x}(t)\cdot \pi_{x\cdot y}(s))(y,t)\\
&=\big((y\cdot x)y, (\pi_{y\cdot x}(t)\cdot \pi_{x\cdot y}(s))\pi_{y\cdot x}(t)\big).
\end{align*}

Therefore, by Lemma \ref{L:xyx}, $\Pi$ can be extended a group homomorphism, still denoted by $\Pi$, 
from $G_{R_X\ltimes_\pi R_S}$ to $G_{R_X}\ltimes_\pi G_{R_S}$. 
\end{proof}

In general, the homomorphism $\Pi$ obtained in Proposition \ref{P:semidp} is not an isomorphism. For instance, if  $R_X$ and $R_S$ are the trivial YBE
solutions, and $\pi$ is the trivial action of $X$ on $S$, then
$G_{R_X\ltimes_\pi R_S}\cong \bZ^{X\times S}$ by Corollary \ref{C:semi}, 
while 
$G_{R_X}\ltimes_\pi G_{R_S}=G_{R_X}\times G_{R_S}\cong \bZ^{X\sqcup S}$.

\end{document}